\documentclass[reqno,12pt,draft]{amsart}
\usepackage{amsmath,amsthm,amssymb}
\usepackage{amsfonts}

\newcommand{\RN}[1]{%
  \textup{\uppercase\expandafter{\romannumeral#1}}%
}

\date{}

\newtheorem{theorem}{Theorem}[section]

\newtheorem{lemma}[theorem]{Lemma}

\newtheorem{proposition}[theorem]{Proposition}

\newtheorem{remark}[theorem]{Remark}
\numberwithin{equation}{section}

\begin{document}

\centerline{{\bf A note on Lusin-type approximation of  Sobolev functions}}
\centerline{{\bf on Gaussian spaces}}

\vskip .2in

\centerline{{\bf  Alexander Shaposhnikov$^{a,}$\footnote{Faculty of Mathematics, University of Bielefeld, D-33615 Bielefeld, Germany, e-mail: shal1t7@mail.ru, ashaposh@math.uni-bielefeld.de}}}

\vskip .2in
\centerline{{\bf Abstract}}

\vskip .1in
We establish
new approximation results in the sense of Lusin for Sobolev functions $f$ with $|\nabla f| \in L\log L$
on infinite-dimensional spaces equipped with  Gaussian measures.
The proof relies on some new pointwise estimate for the approximations based on the corresponding semigroup which can be of independent interest.

AMS Subject Classification: 60H07, 28C20, 46E35

Keywords: Sobolev class, Gaussian measure, Ornstein--Uhlenbeck semigroup, Lusin-type approximation

\vskip .2in

\section{Introduction}

We say that a function $f:X\rightarrow\mathbb{R}$ on a metric measure space $(X, d, m)$ 
is approximable in the sense of Lusin by Lipschitz functions if for any given $\varepsilon > 0$ there exists a
Lipschitz function $g: X \rightarrow \mathbb{R}$ and a Borel set
$S \subset X$ such that $m(X \setminus S) < \varepsilon$ and
$f \equiv g$ on $S$. A~quantitative version of this property can be formulated as follows:
\begin{equation}\label{ineq:q_lusin}
|f(x) - f(y)| \leq d(x, y)(g(x) + g(y))
\end{equation}
for some measurable nonnegative function $g$ and
$x, y \in X\setminus N$, where $N$ is a Borel set of $m$-measure zero.
For Sobolev or BV functions on
$\mathbb{R}^d$ equipped with the standard Lebesgue measure
$\lambda$,
F.C. Liu (see \cite{liu}) obtained the following important result:
$$
|f(x) - f(y)| \leq |x  - y|(M(x) + M(y)),
$$
$$
M(x) := C_d \sup_{r > 0}\frac{1}{\lambda(B(x, r))} \int_{B(x, r)}|\nabla f|\,d\lambda,
$$
see also the book \cite{Ziemer} for a detailed discussion of this problem.
In particular, for $p > 1$ and a function $f$ from the Sobolev class
$W^{1, p}$ in the inequality (\ref{ineq:q_lusin}) one can choose $g \in L^{p}$. Moreover, it is well-known 
that in the class of metric measure spaces $(X, d, m)$ satisfying
the doubling and $1$-Poincar\'e inequality, the property (\ref{ineq:q_lusin})
with $g \in L^{p}$ characterizes the class $W^{1,p}$,
while for general metric measure structures it is the basis of the definition of the so-called Hajlasz Sobolev functions, see e.g. \cite{Heinonen}. For metric measure spaces without the doubling property the classical finite dimensional arguments are not available anymore.
In the recent paper \cite{AmrosioBrueTrevisan} L. Ambrosio, E. Brue, D. Trevisan put forward an alternative approach based on some estimates
for heat semigroups which applies to Gaussian and $RCD(K, \infty)$ spaces. For Sobolev functions $f \in  W^{1, p}$, $p > 1$ on the Wiener space $(\mathcal{W}, \mathcal{H}, m)$ the result from \cite{AmrosioBrueTrevisan} reads as follows: there exists a version of
the function $f$, such that
$$
|f(x) - f(y)| \leq |x - y|_{\mathcal{H}}(M(x) + M(y)),
$$
$$
M(x) := C\Bigl(
\sup_{t > 0}T_{t}|\nabla_{\mathcal{H}} f|(x) +
\sup_{t > 0}T_{t}|\sqrt{1 - L}f|(x)\Bigr),
$$
where $\{T_{t}\}$ is the standard Ornstein-Uhlenbeck semigroup,
$L$ is its generator and $C$ is some universal constant.
In the Gaussian setting for $p > 1$ this theorem provides a natural
counterpart of the classical finite-dimensional results, however, for
$p = 1$ the arguments break down for multiple reasons. First, for
the Sobolev space $W^{1, 1}$ Meyer's equivalence is not available anymore, in particular,
$f$ from $W^{1, 1}$ does not need to belong to the domain of the operator $\sqrt{1 - L}$ in $L^1$.
Second, even for functions $f$ with $\sqrt{1 - L}f$ in $L^1$ it is  unknown if the maximal function
$$
\sup_{t > 0}T_{t}|\sqrt{1 - L}f|$$
is finite $m$-a.e. or not. In particular, whether the maximal operator
$$\sup_{t > 0}T_{t}g, \ g\in L^1(X, m)$$
is of weak $(1, 1)$ type in the infinite-dimensional case has been an open problem for a long time.
In this paper we construct a Lusin-type approximation for Sobolev functions $f $ with $|\nabla f| \in L\log L$ on Gaussian spaces using a modification of the approach from \cite{AmrosioBrueTrevisan}.
Our starting point is the result by I. Shigekawa \cite{Shigekawa} which gives a weaker form of Meyer's equivalence for $L^1$:
$$
\|\sqrt{-L}f\|_1 \leq C\| \nabla f \|_{L\log L}, \ \|\nabla f\|_1 \leq C\| \sqrt{-L}f \|_{L\log L}
$$
To overcome the lack of the weak bound for the maximal operator in
$L^1$ we modify the smoothing procedure. This enables us to obtain a dimension-independent bound using the classical Hopf--Dunford--Schwartz maximal inequality that might be of independent interest. The paper is organized as follows. Section 2 contains the main abstract semigroup--theoretic tools.
In Section 3 we discuss Shigekawa's bound and extend it to more general Ornstein-Uhlenbeck semigroups. Section 4 contains the main results.

\section{Abstract semigroup-theoretic results}

Let $(X, \mathcal{F}, m)$ be an abstract measure space,
where $m$ is a probability measure. Let $\{T_t\}$ be a symmetric
Markov semigroup acting on $L^2(X, \mathcal{F}, m)$ and let $L$
be its generator.
A semigroup of this class has a canonical extension to a contraction semigroup on all $L^p(X, \mathcal{F}, m)$ spaces.
For $p \in [1, \infty)$ and $f \in L^p(X, \mathcal{F}, m)$
we write $f \in D_{p}(\sqrt{-L})$ if there exists a sequence
$(f_n) \subset D(\sqrt{-L}) \cap L^p(X, \mathcal{F}, m)$
converging to $f$ in $L^p(X, \mathcal{F}, m)$ with
$(\sqrt{-L}f_n)$ converging to some function $g$ in $L^p(X, \mathcal{F}, m)$.
Using the symmetry of $\sqrt{-L}$ it is easy to see that if
the set of functions $h \in D(\sqrt{-L})\cap L^{p'}(X, \mathcal{F}, m)$
is dense in $L^{p'}(X, \mathcal{F}, m)$ in the weak-$*$ topology
(in duality with $L^p(X, \mathcal{F}, m)$) then  $g = \sqrt{-L}f$
is uniquely determined. In our cases of interest, where $\{T_t\}$
is an Ornstein-Uhlenbeck semigroup, the required density can be easily verified explicitly.

The next proposition is the classical Hopf--Dunford--Schwartz maximal inequality which will play the crucial role in the proof of the main results.
\begin{proposition}\label{pr:hopf}
Let $f \in L^{1}(X, \mathcal{F}, m)$. Then for any $\lambda > 0$
$$
m\Bigl( x: \sup_{t > 0}\frac{1}{t}\int_{[0, t]}T_{s}f\,ds \geq \lambda\Bigr)
\leq \frac{\|f\|_1}{\lambda}.
$$
\end{proposition}
\begin{proof}
See \cite{DunfordSchwartz}, chapter \RN{8} (6,7).
\end{proof}

Now let us introduce the operators
$\{A_t\}_{t > 0}$ and $\{M_t\}_{t > 0}$ as follows:
$$
A_t := \frac{1}{t}\int_{[t, 2t]}T_s\,ds, \
M_t := \frac{1}{t}\int_{[0, t]}T_s\,ds.
$$
It is easy to see that for a nonnegative function $g \in L^1$
$$
A_t g (x) \leq 2 M_{2t} g (x).
$$
Consequently, for any $f \in L^1$
\begin{equation}\label{eq:trivial_ineq}
\sup_{s > 0} A_{s} |f|  \leq 2 \sup_{s > 0} M_{s} |f|.
\end{equation}

\begin{theorem}\label{th:main_pointwise_approximation}
There exists a universal constant $C > 0$ such that for any
$f \in D_{1}(\sqrt{-L})$
$$
|A_{t}f(x) - f(x)| \leq C\sqrt{t} \sup_{s > 0} M_{s}|\sqrt{-L}f|(x) \
\text{for $m$-a.e.}\ x \in X.
$$
\end{theorem}
\begin{proof}
Using simple density arguments one can see that it is sufficient to prove this estimate just for functions from $D(\sqrt{-L})$. In this case
for the difference $T_{r}f - f$ we have the following classical representation:
$$
T_{r}f - f = \int_{[0, \infty)}K(s, r)T_{s}\sqrt{-L}f\,ds, \ r \geq 0,
$$
where
$$
K(s, r) := \frac{1}{\sqrt{\pi}}\Biggl(
\frac{\chi_{s > r}}{(s - r)^{1/2}} -
\frac{\chi_{s > 0}}{s^{1/2}}
\Biggr),
$$
e.g. see Proposition 2.1 in \cite{AmrosioBrueTrevisan} and also \cite{Stein}.
Then:
\begin{equation}\label{eq:main_representation}
A_{t}f - f  = \int_{[0, \infty)}U(s, t) T_{s}\sqrt{-Lf}\,ds,
\end{equation}
where
$$
U(s, t) := \frac{1}{t}\int_{[t, 2t]}K(s, r)\,dr.
$$

One can see that $K(s, r)$ is not smooth with respect to $s$
at $r$. However, it turns out that the ``averaged'' over $r$
version of $K$ is already absolutely continuous with respect to $s$
for all $s > 0$. Below we will show that there exists a universal constant
$C > 0$ such that for the function
$$
Q(s, t) := s\frac{\partial U}{\partial s}(s, t)
$$
the following inequality holds:
\begin{equation}\label{eq:integral_bound}
\int_{[0, \infty)}|Q(s,t)|\,ds \leq C\sqrt{t}.
\end{equation}
Now let us prove the inequality (\ref{eq:integral_bound}).
It is easy to see that the functions $K, U$ and $Q$
have the following homogeneity property: for any $a > 0$
$$
K(as, ar) = \frac{1}{\sqrt{a}}K(s, r),\
U(as, at) =  \frac{1}{\sqrt{a}}U(s, t),
$$
$$
Q(as, at) = \frac{1}{\sqrt{a}}Q(s, t).
$$
Then:
$$
\int_{[0, \infty)}|Q(s, t)|\,ds =
\int_{[0, \infty)}|Q(ts', t)|t\,ds' =
\sqrt{t}\int_{[0, 1)}|Q(s', 1)|\,ds'.
$$
Consequently, to establish the bound (\ref{eq:integral_bound}) it is sufficient to prove that
$$
\int_{[0, \infty)}|Q(s, 1)|\,ds < \infty.
$$
For $s \in (0, 1)$
$$
U(s, 1) = - \frac{1}{\sqrt{\pi}}\frac{1}{\sqrt{s}}.
$$
For $s \in (1, 2)$
$$
U(s, 1) = - \frac{1}{\sqrt{\pi}}\frac{1}{\sqrt{s}} +
\frac{2}{\sqrt{\pi}}\sqrt{s - 1}.
$$
For $s \in (2, \infty)$
\begin{multline*}
U(s, 1) = -\frac{1}{\sqrt{\pi}}\frac{1}{\sqrt{s}} +
\frac{2}{\sqrt{\pi}}\bigl(\sqrt{s - 1} - \sqrt{s - 2}\bigr) \\
= -\frac{1}{\sqrt{\pi}}\frac{1}{\sqrt{s}} +
\frac{2}{\sqrt{\pi}}\frac{1}{\sqrt{s - 1} + \sqrt{s - 2}} \\
= \frac{1}{\sqrt{\pi}}
\frac{
\sqrt{s} - \sqrt{s - 1} + \sqrt{s} - \sqrt{s - 2}
}{\sqrt{s}(\sqrt{s - 1} + \sqrt{s - 2})}
\end{multline*}
Now it is easy to verify that
$$
Q(s, 1) = s \frac{\partial U}{\partial s}(s, 1)
$$
is integrable on $[0, \infty)$.
Excluding if necessary a set of measure zero
by Proposition
\ref{pr:hopf} we can assume that 
$$
\sup_{s > 0} M_{s}|\sqrt{-L}f|  < \infty.
$$
Applying integration by parts we obtain the following equality:
\begin{multline*}
\int_{[\varepsilon, 1/\varepsilon)}U(s, t)T_{s}\sqrt{-L}f\,ds =
\int_{[\varepsilon, 1/\varepsilon)}U(s, t)
\frac{d}{d s}\Bigl(s M_{s}\sqrt{-L}f\Bigr)\,ds \\
=
- \int_{[\varepsilon, 1/\varepsilon)}\frac{\partial U(s, t)}{\partial s}sM_{s}\sqrt{-L}f\,ds \\+
U(\varepsilon^{-1}, t)\varepsilon^{-1} M_{1/\varepsilon}\sqrt{-L}f - U(\varepsilon, t)\varepsilon M_{\varepsilon}\sqrt{-L}f.
\end{multline*}
Since
$$
U(\varepsilon^{-1}, t) = O(\varepsilon^{3/2}), \
U(\varepsilon, t) = O(\varepsilon^{-1/2}) 
\ \text{as}
\ \varepsilon \to 0,
$$
then
$$
A_{t}f - f = -\int_{[0, \infty)}Q(s,t)M_{s}\sqrt{-L}f\,ds
$$
and the bound (\ref{eq:integral_bound}) yields
the required estimate
$$
|A_{t}f(x) - f(x)| \leq C\sqrt{t} \sup_{s > 0} M_{s}|\sqrt{-L}f|(x) \
\text{for $m$-a.e.}\ x \in X.
$$
\end{proof}

\section{Meyer-type inequality in $L^1$}

Let $m = N_Q$ be a centered Gaussian measure on 
$\mathbb{R}^d$ with the covariance operator $Q > 0$.
We will be concerned with the semigroup given by Mehler's formula (see e.g. \cite{Bogachev})
\begin{multline*}
P_{t}f(x) :=  \int_{H}f(e^{At}x + \sqrt{1 - e^{2At}}y)\,dm(y) \\
=\int_{H}f(e^{At}x + y)dN_{Q_t}(y) = \int_{H}f(y)dN_{e^{At}x, Q_t}(y),
\end{multline*}
where we have set
$$
A := -\frac{1}{2}Q^{-1}, \ Q_t := \int_{[0, t]}e^{2As}\,ds = Q\bigl(1 - e^{2At}\bigr)
$$
and $N_{e^{At}x, Q_t}$ denotes the unique Gaussian measure with mean $e^{At}x$ and covariance $Q_t$.
The generator of $\{P_t\}$ is the Ornstein--Uhlenbeck operator
$$
L := \frac{1}{2}\Delta + \langle Ax, \nabla \rangle
$$
and $N_Q$ is the unique invariant measure of $\{P_t\}$.
Although in this section we assume that the underlying space is finite-dimensional, the final inequalities do not include any dimension-dependent constants and are valid for the infinite-dimensional case as well, this can be justified by the standard approximation arguments.
In the $L^p$ setting Meyer's inequalities establish the equivalence of two kinds of norms on the Sobolev space: one is defined by means of the gradient and the other by means of the square root of the Ornstein--Uhlenbeck operator:
for any $p~\in~(1, \infty)$ there exist
positive constants $C_1, C_2$ such that for all
$f \in C_{0}^{\infty}(\mathbb{R}^d, \mathbb{R})$
$$
C_{1}\|\sqrt{1 - L}f\|_p \leq
\|\nabla f\|_p + \|f\|_p \leq C_{2}\|\sqrt{1 - L}f\|_p.
$$
For the case when $Q$ is the identity matrix this result was obtained in 
\cite{Meyer}, see also \cite{Pisier}. Later it was extended to general  
Ornstein--Uhlenbeck operators in \cite{Chojinowska-Michalik-Goldys}, \cite{Shigekawa92}.
When $p = 1$ this equivalence does not hold anymore, nevertheless, in \cite{Shigekawa} for
the Ornstein--Uhlenbeck operator
$$
\mathcal{L} = \Delta - \langle x, \nabla \rangle
$$
the following inequalities were put forward:
there exist
$C_{1}, C_{2} > 0$ such that for all
$f \in C_{0}^{\infty}$
$$
\|\sqrt{-L}f\|_1 \leq
C_1\|\nabla f\|_{L\log L},
$$
$$
\|\nabla f\|_1 \leq
C_{2}\|\sqrt{-L}f\|_{L\log L}.
$$
Here $L\log L$ is the space of functions such that
$$
\int |f| \log(1 + |f|)\,d\mu < \infty,
$$
equipped with the Orlicz--Luxemburg norm
\begin{equation}\label{def:l_log_l_norm}
\|f\|_{L\log L} := \inf \Bigl\{
\lambda > 0: \ \int \Phi(|f|/\lambda)\,d\mu \leq 1
\Bigr\},
\end{equation}
\begin{equation}\label{def:Phi}
\Phi(a) := \int_{[0, a]}\log(1 + t)\,dt.
\end{equation}
While the reasoning in \cite{Shigekawa} 
is based on some probabilistic arguments
which can be modified to cover the general case, 
below we present a simple analytical approach using the sharp $L^p$ estimates for the Riesz transform obtained in \cite{CaDr}. We thank the anonymous referee for suggesting this approach as it yields a shorter proof and works in a greater generality.
\begin{theorem}\label{th:main_l_log_l}
There exists $C > 0$ such that for any $f \in C_{0}^{\infty}$
$$
\|\sqrt{-L}f\|_{1} \leq C \|\nabla f\|_{L\log L}.
$$
\end{theorem}
\begin{proof}
Without loss of generality we may assume that 
$$
\|\nabla f\|_{L\log L} = 1.
$$
Let $g$ be a fixed $C_{b}^{\infty}$ function such that 
$$
\int g\,dm = 0, \ \|g\|_{\infty} \leq 1.
$$
Corollary~1 of Theorem~2 from \cite{CaDr} for $p > 1$ yields the bound
\begin{equation}\label{eq:main_riesz_l_p}
\|\nabla L^{-1/2}g\|_{p} \leq 12 (p^{*} - 1) \|g\|_{p},
\end{equation}
where 
$$
p^{*} := \max(p, q), \ \frac{1}{p} + \frac{1}{q} = 1.
$$
It is easy to see that these inequalities imply exponential integrability
$$
\int \exp\Bigl(
\frac{1}{36}|\nabla (-L)^{-1/2}g|
\Bigr)\,dm 
\leq \sum_{n = 0}^{\infty}
\frac{\bigl(
n/3
\bigr)^n}
{n!}
< \infty,
$$
where we have used the assumption $\|g\|_\infty \leq 1$.
Finally,
\begin{multline*}
\int g\sqrt{-L}f \,dm = 
\int \langle\nabla (-L)^{-1/2} g, \nabla f\rangle\,dm \\
\leq 
C'\biggl(
\|\nabla f\|_{L\log L} + 
\log 
\int \exp\Bigl(
\frac{1}{36}|\nabla (-L)^{-1/2}g|
\Bigr)\,dm 
\biggr) \leq  C.
\end{multline*}
Now it is trivial to complete the proof.
\end{proof}
The proof of Theorem \ref{th:main_l_log_l} 
essentially depends on the linear growth of the norm of the Riesz transform in 
$L^p$ as $p \to \infty$ and the standard dualization argument.
In turn, the next statement makes use of the assymptotics
of the Riesz transform's norm in $L^p$ as $p \to 1$.
\begin{theorem}
There exists $C > 0$ such that for any $f \in C_{0}^{\infty}$
$$
\|\nabla f\|_{1} \leq C \|\sqrt{-L}f\|_{L\log L}.
$$
\end{theorem}
\begin{proof}
For $p \in (1, 2)$ the estimate \ref{eq:main_riesz_l_p} 
for the norm of the Riesz transform can be formulated as follows:
$$
\|\nabla f\|_{p} \leq \frac{12}{p - 1} \|\sqrt{-L}f\|_{p}.
$$
The required inequality follows by the classical Yano's extrapolation theorem, see e.g. \cite{EdKr},  \cite{Yano}.
\end{proof}

\section{Main results}
\subsection{Sobolev functions on the Wiener space}
In this section we consider a Wiener space $(\mathcal{W}, \mathcal{H}, \mu)$, i.e.
$\mathcal{W}$ is a separable Banach space equipped with a centered Gaussian measure $\mu$ which is not concentrated on a closed proper linear subspace of $\mathcal{W}$, $\mathcal{H}$ is its Cameron--Martin space, see e.g. \cite{Bogachev}, \cite{Ustunel}.
Let us recall that a Borel probability measure $\mu$ on $\mathcal{W}$
is called centered Gaussian if every continuous linear functional
$l \in \mathcal{W}^{*}$
is a centered Gaussian random variable on $(\mathcal{W}, \mu)$, i.e.
$$
\int_{\mathcal{W}}\exp(il)\,d\mu  = \exp\Bigl(-\frac{1}{2}\int_{\mathcal{W}}l^2\,d\mu\Bigr).
$$
The Cameron--Martin space $\mathcal{H} = \mathcal{H}(\mu)$ of this measure is the set
of all vectors $h \in W$ with $|h|_{\mathcal{H}} < \infty$, where
$$
|h|_{\mathcal{H}} = \sup\bigl\{l(h): \ l \in \mathcal{W}^{*}, \|l\|_{L^2(\mu)} \leq 1\bigr\}.
$$
This is also the set of all vectors the shifts along which give measures equivalent to $\mu$.
The nondegeneracy of $\mu$ means that it is not concentrated on any proper linear subspace of $\mathcal{W}$.
It is known (see e.g. \cite{Bogachev}) that in this case $(\mathcal{H}, |\cdot|_{\mathcal{H}})$ is a separable Hilbert space densely embedded into $\mathcal{W}$.
Let  $\mathcal{F}\mathcal{C}_b^\infty$ denote the
class of all functions on $\mathcal{W}$ of the form
$$
f(x)=f_0(l_1(x),\ldots,l_n(x)), \quad f_0\in C_{b}^\infty (\mathbb{R}^n),
l_i\in \mathcal{W}^{*}.
$$
The gradient $\nabla_{\mathcal{H}}f$ of $f\in \mathcal{F}\mathcal{C}_b^\infty$ along the subspace $\mathcal{H}$ is defined by the equality
$$
(\nabla_{\mathcal{H}}f(x),h)_{\mathcal{H}}=\partial_h f(x).
$$
The Sobolev space $W^{1, p}$ (see \cite{Bogachev}, \cite{B10})
is defined as the completion  of $\mathcal{F}\mathcal{C}_b^\infty$
in the norm
$$
\|f\|_{1, p} := \| \nabla f_{\mathcal{H}} \|_{p} + \|f\|_{p},
$$
where $\|\cdot\|_p$ denotes the norm in $L^{p}(\mu)$.
In this context the Ornstein--Uhlenbeck semigroup $\{T_t\}$
is given by Mehler's formula
$$
T_{t}f (x) := \int_{W}f(e^{-t} x+ \sqrt{1-e^{-2t}}y)\,d\mu(y)
$$
and its generator is the standard Ornstein--Uhlenbeck operator $L$
$$
L := \Delta - \langle x, \nabla_{\mathcal{H}}\rangle.
$$
Similarly to $W^{1, p}$  one can define the space $W^{1, L\log L}$, where
the norm from $L^{p}(\mu)$ is replaced with the Orlicz norm
$L\log L$ (\ref{def:l_log_l_norm}):
$$
\|f\|_{1, L\log L} := \| \nabla f_{\mathcal{H}} \|_{L\log L} + \|f\|_{L\log L}.
$$
Alternatively, one can describe
the class $W^{1, L\log L}$ as a linear subspace of $W^{1,1}$
consisting of the functions for which the norm of the gradient
$\|\nabla f\|_{L\log L}$ is finite. 

Now let us recall the log-convexity property of the semigroup $\{T_{t}\}$
which will play an important role below.
\begin{lemma}\label{le:log_convex}
For every nonnegative Borel function $g \in L^1$ and every $t > 0$
the map $\log T_{t}g$, where $T_{t}g$ is defined by Mehler's formula, is $-\frac{1}{t}$-convex with respect to the Cameron--Martin distance, i.e.
$$
T_{t}g((1 - s)x_0 + sx_{1}) \leq
\exp\Biggl\{\frac{s(1-s)}{2t}|x_1 - x_0|^{2}_{\mathcal{H}}\Biggr\}
(T_{t}g(x_0))^{1 - s} (T_{t}g(x_1))^{s}.
$$
for every $x_0, x_1 \in \mathcal{W}$ with $x_0 - x_1 \in \mathcal{H}$
and $s \in [0, 1]$.
\end{lemma}
\begin{proof}
See e.g. Lemma 3.4 in \cite{AmrosioBrueTrevisan} or
Lemma 5.14 in \cite{B_OU}.
\end{proof}
It is worth noting that in Lemma \ref{le:log_convex}
it is important that we  consider the version of $T_t g$
given by Mehler's formula in the  {\it pointwise} sense. This is possible since $g$ is nonnegative and Borel.

\begin{lemma}
Let $g$ be a nonnegative Borel function in $L^1$ and $t > 0$.
Then for the function
$$
A_{t}g  = \frac{1}{t}\int_{[t, 2t]}T_{s}g\,ds,
$$
where $T_{s}g$ is defined by Mehler's formula,
and every $x_0, x_1 \in W$ with $x_0 - x_1 \in H$, $s \in [0, 1]$
the following inequality holds:
$$
A_{t}g((1 -s)x_0 + sx_1) \leq
\exp\Biggl\{\frac{s(1-s)}{2t}|x_1 - x_0|^{2}_{\mathcal{H}}\Biggr\}
(A_{t}g(x_0))^{1 - s}(A_{t}g(x_1))^s.
$$
\end{lemma}
\begin{proof}
Indeed, the required bound follows from Lemma \ref{le:log_convex}
and the standard H\"older's inequality:
\begin{multline*}
A_{t}g((1 -s)x_0 + sx_1) =
\frac{1}{t}\int_{[t, 2t]}T_{u}g((1 -s)x_0 + sx_1)\,du\\
\leq
\frac{1}{t}\int_{[t, 2t]}
\exp\Biggl\{\frac{s(1-s)}{2u}|x_1 - x_0|^{2}_{\mathcal{H}}\Biggr\}
(T_{u}g(x_0))^{1 - s} (T_{u}g(x_1))^{s}\,du\\
\leq
\exp\Biggl\{\frac{s(1-s)}{2t}|x_1 - x_0|^{2}_{\mathcal{H}}\Biggr\}
\frac{1}{t}\int_{[t, 2t]}(T_{u}g(x_0))^{1 - s} (T_{u}g(x_1))^{s}\,du\\
\leq
\exp\Biggl\{\frac{s(1-s)}{2t}|x_1 - x_0|^{2}_{\mathcal{H}}\Biggr\} \\
\times \Biggl[
\frac{1}{t}\int_{[t, 2t]}T_{u}g(x_0)\,du
\Biggr]^{1 -s}
\Biggl[
\frac{1}{t}\int_{[t, 2t]}T_{u}g(x_1)\,du
\Biggr]^{s}
\end{multline*}
\end{proof}

\begin{lemma}\label{le:lip}
Let $f \in W^{1, 1}$.
Then there exists a Borel set $\Omega_{f}$
with $\mu(\Omega_{f}) = 1$ such that for any $t > 0$,
$x_0, x_1 \in \Omega_{f}$ with $x_0 - x_1 \in \mathcal{H}$
$$
|A_{t}f(x_1) - A_{t}f(x_0)| \leq
|x_1 - x_0|_{\mathcal{H}} e^{\frac{|x_1 - x_0|_{\mathcal{H}}^2}{4t}}\bigl(
A_{t}|\nabla_{\mathcal{H}} f|(x_1) +
A_{t}|\nabla_{\mathcal{H}} f|(x_0)
\bigr).
$$
\end{lemma}
\begin{proof}
Let $h := x_1 - x_0$, $h \in \mathcal{H}$.
We first assume that $t > 0$ is fixed and $f$ is a smooth cylindrical function.
Then
\begin{multline*}
|A_{t}f(x_1) - A_{t}f(x_0)| =
\Biggl| \int_{[0,1]} \langle \nabla_{\mathcal{H}} A_{t}f((1-s)x_0 + sx_1), h\rangle\,ds\Biggr| \\
\leq |h| \int_{[0,1]} | \nabla_{\mathcal{H}} A_{t}f((1-s)x_0 + sx_1)|\,ds,
\end{multline*}
\begin{multline*}
|\nabla_{H} A_{t}f((1-s)x_0 + sx_1)| =
\Bigl|
\frac{1}{t}\int_{[t,2t]} \nabla_{\mathcal{H}}T_{u}f((1-s)x_0 + sx_1)\,du
\Bigr| \\
=
\Bigl|
\frac{1}{t}\int_{[t,2t]} e^{-u}T_{u}\nabla_{\mathcal{H}}f((1-s)x_0 + sx_1)\,du
\Bigr|
\\
\leq
\frac{1}{t}\int_{[t,2t]}|T_{u}\nabla_{\mathcal{H}}f((1-s)x_0 + sx_1)|\,du \\
= A_{t} |\nabla_{\mathcal{H}}f|((1-s)x_0 + sx_1)\\
 \leq
e^{\frac{|x_1 - x_0|^2}{4t}}
(A_{t}|\nabla_{\mathcal{H}}f|(x_0))^{1 - s}
(A_{t}|\nabla_{\mathcal{H}}f|(x_1))^{s}
\\
\leq
e^{\frac{|x_1 - x_0|^2}{4t}}\bigl(
A_{t}|\nabla_{\mathcal{H}}f|(x_0) + A_{t}|\nabla_{\mathcal{H}}f|(x_1)
\bigr).
\end{multline*}
Therefore,
$$
|A_{t}f(x_1) - A_{t}f(x_0)| \leq
|x_1 - x_0|_H e^{\frac{|x_1 - x_0|^2}{4t}}\bigl(
A_{t}|\nabla_{\mathcal{H}} f|(x_1) +
A_{t}|\nabla_{\mathcal{H}} f|(x_0)
\bigr).
$$
Now for a given $f \in W^{1,1}$ let us find a sequence of smooth cylindrical functions $(f_n)$ converging to $f$ in $W^{1,1}$.
It is easy to see that
by passing to a subsequence we may assume that
$A_{t}f_n$, $A_{t}|\nabla_{\mathcal{H}} f_n|$ converge to
$A_{t}f$ and $A_{t}|\nabla_{\mathcal{H}} f|$ respectively
in $L^1$ and on some set $\Omega_{f, t}$ of full measure.
Then for any $x_0, x_1 \in \Omega_{f, t}$ with $x_1 - x_0 \in \mathcal{H}$ we have 
\begin{multline*}
|A_{t}f(x_1) - A_{t}f(x_0)| = \lim_{n\to\infty}
|A_{t}f_n(x_1) - A_{t}f_n(x_0)| \\
\leq
\lim_{n\to\infty} |x_1 - x_0|_{\mathcal{H}}
e^{\frac{|x_1 - x_0|^2}{4t}}\bigl(
A_{t}|\nabla_{\mathcal{H}} f_n|(x_1) +
A_{t}|\nabla_{\mathcal{H}} f_n|(x_0)
\bigr)
\\
=
|x_1 - x_0|_{\mathcal{H}}
e^{\frac{|x_1 - x_0|^2}{4t}}\bigl(
A_{t}|\nabla_{\mathcal{H}} f|(x_1) +
A_{t}|\nabla_{\mathcal{H}} f|(x_0).
\end{multline*}
Now similarly to the proof of Theorem \ref{th:main_pointwise_approximation} we can notice that there exists
a set $\Omega'_{f}$ of full measure such that for every $x \in \Omega'_{f}$ the mappings
$$
t \mapsto A_{t}(x), \ t \mapsto A_{t}|\nabla_{\mathcal{H}} f|(x)
$$
are continuous on $(0, \infty)$. It easy to see that for any
$x_0, x_1 \in \Omega_f$ with $x_1 - x_0 \in \mathcal{H}$ and any $t > 0$
$$
|A_{t}f(x_1) - A_{t}f(x_0)| \leq
|x_1 - x_0|_{\mathcal{H}} e^{\frac{|x_1 - x_0|^2}{4t}}\bigl(
A_{t}|\nabla_{\mathcal{H}} f|(x_1) +
A_{t}|\nabla_{\mathcal{H}} f|(x_0)
\bigr),
$$
where
$$
\Omega_f := \Omega'_f \cap \bigcap_{t_i \in \mathbb{Q} \cap [0, \infty)}\Omega_{f,t_i}.
$$
\end{proof}

\begin{lemma}\label{le:wiener_l_log_l}
Let $f \in W^{1,L\log L}$, i.e.
$$
\int_{\mathcal{W}}|\nabla_{\mathcal{H}}f(x)|\log(1 + |\nabla_{\mathcal{H}} f(x)|)\,d\mu(x) < \infty.
$$
Then $f \in D_{1}(\sqrt{-L})$ and
$$
\|\sqrt{-L}f\|_1 \leq C\|\nabla_{\mathcal{H}}f\|_{L\log L},
$$
where $C$ is some positive constant which does not depend on $f$.
\end{lemma}
\begin{proof}
For a function $f \in \mathcal{FC}^{\infty}_{b}$ this statement easily follows by Theorem 1.1 from \cite{Shigekawa}, 
this is also a particular case of Theorem \ref{th:main_l_log_l} from Section~3.
The general case follows from the standard approximation arguments since smooth cylindrical functions are dense in $W^{1,L\log L}$.
\end{proof}

\begin{lemma}\label{le:approximation_by_a}
Let $f \in W^{1,L\log L}$. There exist a universal constant $C > 0$
and a set $\Omega_f$ with $\mu(\Omega_f) = 1$
such that for any $t > 0$
$$
|A_{t}f(x) - f(x)| \leq C\sqrt{t}\sup_{s > 0}M_{s}|\sqrt{-L}f|(x), \, x \in \Omega_{f}.
$$
\end{lemma}
\begin{proof}
This follows immediately by Theorem \ref{th:main_pointwise_approximation} and Lemma \ref{le:wiener_l_log_l}.
\end{proof}

The next theorem is our main result for Sobolev functions on the Wiener space.
\begin{theorem}\label{th:main_wiener_space_theorem}
Let $f \in W^{1,L\log L}$.
There exist
a set $\Omega_f$ with $\mu(\Omega_f) = 1$ and
a universal constant $C > 0$
such that for any $x_0, x_1 \in \Omega_f$ with $x_1 - x_0 \in \mathcal{H}$
$$
|f(x_1) - f(x_0)| \leq C |x - y|_{\mathcal{H}}(M(x_0) + M(x_1)),
$$
where
$$
M(x) := \sup_{t > 0}\frac{1}{t}\int_{[0, t]}T_{s}|\sqrt{-L}f|(x)\,ds +
\sup_{t > 0}\frac{1}{t}\int_{[0, t]}T_{s}|\nabla_{\mathcal{H}} f|(x).
$$
\end{theorem}
\begin{proof}
Let $\Omega_f$ be the intersection of the sets of full measure provided by
Lemma \ref{le:lip} and Lemma \ref{le:approximation_by_a}.
For any $x_0, x_1 \in \Omega_f$ and any $t > 0$
\begin{multline*}
|f(x_1) - f(x_0)|\\
 \leq |f(x_1) - A_{t}f(x_1)| +
|A_{t}f(x_1) - A_{t} f(x_0)| +  |f(x_0) - A_{t}f(x_0)|\\
\leq C\sqrt{t}\sup_{s > 0}M_{s}f(x_1) + C\sqrt{t}\sup_{s > 0}M_{s}f(x_0) \\
+
|x_1 - x_0|_{\mathcal{H}}e^{\frac{|x_1 - x_0|_{\mathcal{H}}^2}{4t}}\bigl(
A_{t}|\nabla_{\mathcal{H}}f|(x_1) + A_{t}|\nabla_{\mathcal{H}}f|(x_0)
\bigr).
\end{multline*}
It is easy to see that picking $t := |x_1 - x_0|^{2}_{\mathcal{H}}$ 
and taking into account inequality \ref{eq:trivial_ineq}
yields the required estimate.
\end{proof}

\begin{theorem}\label{th:lusin_wiener_space}
Let $f \in W^{1,L\log L}$. Then for every $\varepsilon >0$
there exists an $\mathcal{H}$-Lipschitz  $\mu$-measurable function $g_\varepsilon$, i.e.
$$
|g_{\varepsilon}(x_1) - g_{\varepsilon}(x_0)| \leq C_{\varepsilon}|x_1 - x_0|_{H},  \ x_0, x_1 \in \mathcal{W}, \ x_1 - x_0 \in \mathcal{H}
$$
such that
$$
\mu\bigl(x:\ g_{\varepsilon}(x) \neq f(x)\bigr) \leq \varepsilon.
$$
\end{theorem}
\begin{proof}
Applying the Hopf--Dunford--Schwartz maximal inequality (see Proposition \ref{pr:hopf}) to the semigroup $\{T_t\}$ and the integrable functions
$|\sqrt{-L}f|$ and  $|\nabla_{\mathcal{H}}f|$ yields that for every $\lambda > 0$
$$
\mu \bigl(x: CM(x) \geq \lambda \bigr) \leq
C' \frac{\|\sqrt{-L}f\|_1 +\|\nabla_{\mathcal{H}}f\|_1 }{\lambda}
\leq C'' \frac{\|\nabla_{\mathcal{H}} f\|_{L\log L}}{\lambda},
$$
where $C$ is the constant from Theorem \ref{th:main_wiener_space_theorem}.
Let us choose
$$
\lambda := \frac{1}{\varepsilon C''\|\nabla_{\mathcal{H}} f\|_{L\log L}}
$$
and set
$$
\Omega_{f, \varepsilon} :=  \bigl\{x: CM(x) \leq \lambda \bigr\}.
$$
Then:
$$
\mu(W \setminus \Omega_{f, \varepsilon}) \leq \varepsilon
$$
and for any $x_0, x_1 \in \Omega_{f, \varepsilon}$
$$
|f(x_0) - f(x_1)| \leq \lambda |x_0 - x_1|_{\mathcal{H}}.
$$
Now we can apply the result from \cite{UstunelZakai} to
the function
$
\left.f\right|_{\Omega_{f, \varepsilon}}
$
and obtain a measurable $\mathcal{H}$-Lipschitz function $g_{\varepsilon}$
defined on the whole space $\mathcal{W}$ such that
$$
\left.g\right|_{\Omega_{f, \varepsilon}} = \left.f\right|_{\Omega_{f, \varepsilon}}.
$$
It is clear that by construction
$$
\mu\bigl(x:\ g_{\varepsilon}(x) \neq f(x)\bigr) \leq \varepsilon.
$$
\end{proof}

\begin{remark}
As it is clear from the proofs,
the statements of Theorem~\ref{th:main_wiener_space_theorem}
and Theorem~\ref{th:lusin_wiener_space} remain valid for any function
$$f \in W^{1,1} \cap D_{1}(\sqrt{-L}).$$
However, the case of $f \in W^{1,1}$ or $f \in BV$ when the underlying space $W$ is infinite-dimensional is still open, see also the discussion of this problem in \cite{AmrosioBrueTrevisan}.
\end{remark}
\begin{remark}
In the paper \cite{Alberti} G. Alberti proved that any Borel vector field on
$\mathbb{R}^d$ coincides with the
gradient of some $C^{1}$ function outside of a set of
arbitrarily small Lebesgue measure. This result was extended to the Wiener space setting in \cite{Shaposhnikov10}.
\end{remark}

\subsection{Da Prato's Sobolev spaces}
We refer to the book \cite{Daprato} (see also \cite{Bogachev}, \cite{B10}) for a detailed introduction into this topic. In this setting
the underlying space $\mathcal{W} = H$ is a separable Hilbert space, $m$
is a centered Gaussian measure which  is not concentrated on a closed proper linear subspace of $H$. We denote by
$Q$ the covariance operator associated with $m$. It is well-known
(see \cite{Daprato}, \cite{Bogachev}) that in this case $Q$ is a
nonnegative symmetric operator with finite trace. The Cameron--Martin space of $m$ will be denoted by $\mathcal{H}$. In fact,
$\mathcal{H}$ coincides with the range of $Q^{\frac{1}{2}}$ and
moreover $\| x \|_{\mathcal{H}} = |Q^{-\frac{1}{2}}x|$.
Using the Hilbertian structure of the underlying space $H$ we can introduce the Sobolev spaces $W^{1, p}(H, m)$ obtained as the closure
of smooth cylindrical functions with respect to the norm
$$
\|f\|_{1, p} := \|\nabla f\|_{p} + \|f\|_p.
$$
The difference with the Sobolev classes on the Wiener space which were considered in the previous subsection is that here the gradient
with respect to the Hilbertian structure of the underlying space $H$
is involved rather than with respect to the structure of  the Cameron--Martin space
$\mathcal{H}$. In this context the natural semigroup is given by
a Mehler-type formula
\begin{multline*}
P_{t}f(x) :=  \int_{H}f(e^{At}x + \sqrt{1 - e^{2At}}y)\,dm(y) \\
=\int_{H}f(e^{At}x + y)dN_{Q_t}(y) = \int_{H}f(y)dN_{e^{At}x, Q_t}(y),
\end{multline*}
where we have set
$$
A := -\frac{1}{2}Q^{-1}, \ Q_t := \int_{[0, t]}e^{2As}\,ds = Q\bigl(1 - e^{2At}\bigr)
$$
and $N_{e^{At}x, Q_t}$ denotes the the unique Gaussian measure with mean $e^{At}x$ and covariance $Q_t$.
In this section we denote by $L$ the generator of the semigroup $\{P_t\}$:
$$
L := \frac{1}{2}\Delta + \langle Ax, \nabla \rangle.
$$
We will also assume that for the operator $A$ the following bound holds:
$$
A \leq -\beta,
$$
where $\beta$ is a positive constant.
It is easy to see that in the infinite-dimensional setting we still have the commutation
identity
$$
\nabla P_{t}f  = e^{-At}P_{t}\nabla f.
$$
Consequently, in this case
$$
|\nabla P_{t}f| \leq e^{-\beta t}P_{t}|\nabla f|.
$$
Similarly to the case of the abstract Wiener space we can introduce
the Sobolev class $W^{1, L\log L}(H, m)$ with the norm
$$
\|f\|_{1, L\log L} := \|\nabla f\|_{L\log L} + \|f\|_{L\log L}.
$$
Now we can make use of the extension of Shigekawa's bound established in Section 3 and obtain the  natural counterpart of Lemma \ref{le:wiener_l_log_l} for Da Prato's spaces.
\begin{lemma}\label{le:da_prato_l_log_l}
Let $f \in W^{1,L\log L}(H, m)$, i.e.
$$
\int_{W}|\nabla f(x)|\log(1 + |\nabla f(x)|)\,d\mu(x) < \infty.
$$
Then $f \in D_{1}(\sqrt{-L})$ and
$$
\|\sqrt{-L}f\|_1 \leq C\|\nabla f\|_{L\log L}.
$$
\end{lemma}
\begin{proof}
For a function $f \in \mathcal{FC}^{\infty}_{b}$ this statement follows by Theorem \ref{th:root_upper_bound} from Section~3.
The general Sobolev case is again handled by the standard approximation arguments using the density of smooth cylindrical functions in $W^{1,L\log L}(H, m)$.
\end{proof}
The rest of our intermediate steps work the same as in the Wiener space setting. Therefore, let us conclude this section with the formulation of the final results.
\begin{theorem}\label{th:main_da_prato_theorem}
Let $f \in W^{1,L\log L}(H, m)$.
There exist
a set $\Omega_f$ with $m(\Omega_f) = 1$ and
a universal positive constant $C$
such that for any $x_0, x_1 \in\Omega_f$
$$
|f(x_1) - f(x_0)| \leq C |x_1 - x_0|(M(x_0) + M(x_1)),
$$
where
$$
M(x) := \sup_{t > 0}\frac{1}{t}\int_{[0, t]}P_{s}|\sqrt{-L}f|(x)\,ds +
\sup_{t > 0}\frac{1}{t}\int_{[0, t]}P_{s}|\nabla f|(x).
$$
\end{theorem}
As a by-product we obtain a Lusin-type approximation for functions
from Da-Prato's Sobolev class analogous to Theorem \ref{th:lusin_wiener_space}. The key difference with the case of the Wiener space is that here the Lipschtiz functions with respect to the norm of the underlying Hilbert space are involved rather than the functions which are Lipschtiz-continuous along the Cameron--Martin space $\mathcal{H}$.
\begin{theorem}\label{th:lusin_da_prato}
Let $f \in W^{1,L\log L}(H, m)$. Then for every $\varepsilon >0$
there exists a Lipschitz  $m$-measurable function $g_\varepsilon$
such that
$$
\mu\bigl(x:\ g_{\varepsilon}(x) \neq f(x)\bigr) \leq \varepsilon.
$$
\end{theorem}

\begin{remark}
The statements of Theorem \ref{th:main_da_prato_theorem}
and Theorem \ref{th:lusin_da_prato} remain valid for any function
$$f \in W^{1,1}(H, m) \cap D_{1}(\sqrt{-L}).$$
\end{remark}

\vskip .1in

\centerline{{\bf\it Acknowledgment}}

I would like to thank V.I. Bogachev and M. R\"ockner for fruitful discussions and comments, I am also grateful to L. Ambrosio for some useful information on this subject.
Part of this work was done during a visit to the
Max Planck Institute for Mathematics in the Sciences
in Leipzig. I would like to thank F. Otto and B. Gess for their hospitality
and useful remarks.

This research is supported by the Russian Science Foundation Grant 17-11-01058
at Lomonosov Moscow State University.

\end{document}